\documentclass[a4paper,11pt]{article}

\usepackage{amsmath}
\usepackage{amsthm}
\usepackage{amssymb}
\usepackage[latin1]{inputenc}
\usepackage[english]{babel}

\theoremstyle{plain}
\newtheorem{thm}{Theorem}[section]
\newtheorem{prop}[thm]{Proposition}
\newtheorem{lem}[thm]{Lemma}
\newtheorem{cor}[thm]{Corollary}

\theoremstyle{definition}

\theoremstyle{remark}

\DeclareMathOperator{\tr}{tr}

\DeclareMathOperator{\End}{End}
\DeclareMathOperator{\grad}{grad}

\DeclareMathOperator{\scal}{scal}

\title{A note on Berezin-Toeplitz quantization of the Laplace operator}
\author{Alberto Della Vedova \footnote{
Universit\`a degli Studi di Milano-Bicocca. Dipartimento di Matematica e Applicazioni. Via Cozzi, 53 - 20125 Milano (IT). \textit{alberto.dellavedova@unimib.it}}
}

\begin{document}
\maketitle

\begin{abstract}
Given a Hodge manifold, it is introduced a self-adjoint operator on the space of endomorphisms of the global holomorphic sections of the polarization line bundle. Such operator is shown to approximate the Laplace operator on functions when composed with Berezin-Toeplitz quantization map and its adjoint up to an error which tends to zero when taking higher powers of the polarization line bundle.
\end{abstract}

\section{Introduction}


Let $M$ be a $n$-dimensional projective manifold and let $g$ be a Hodge metric on $M$.
This means that $M$ is equipped with a complex structure $J$ and with a positive Hermitian line bundle $(L,h)$.
Denoted by $\Theta$ the curvature of the Chern connection, the form $\omega = 2\pi i\Theta$ is positive, and it holds $g(u,v) = \omega(u,Jv)$.
Let $$\Delta: C^\infty(M) \to C^\infty(M)$$
be the positive Laplacian associated with the metric $g$ (recall that it is defined by $\Delta (f) \omega^n = - n \, i \partial \bar \partial f \wedge \omega^{n-1}$ for any complex-valued smooth function $f$ on $M$).
In this note it will be shown that $\Delta$ is approximated in a suitable sense by a sequence of self-adjoint positive operators 
$$\Delta_m: V_m \to V_m$$ 
acting on finite dimensional Hermitian vector spaces $V_m$ (see definitions at Sections \ref{sec::V_m} and \ref{sec::Delta_m}).
To be a little more precise, it will be proved that there exist maps 
$$ T_m: C^\infty(M) \to V_m, \qquad T^*_m : V_m \to C^\infty(M),$$ in fact adjoint to each other with respect to suitable Hermitian products, such that 
\begin{equation}\label{maineq}
T^*_m \circ \Delta_m \circ T_m(f) = m^{n-1} \Delta f + O(m^{n-2})
\end{equation}
as $m\to \infty$ for any given smooth function $f$.
For any $m>0$ the map $T_m$ is the well known Berezin-Toeplitz quantization map, and the operator $\Delta_m$ depends only on the projective geometry of the Kodaira embedding of $M$ via $L^m$.
Moreover $\Delta_m$ is related to the metric $g$ via the Fubini-Study metric induced by the $L^2$-inner product on the space of global holomorphic sections of $L^m$ (see Section \ref{sec::Delta_m}).
Thanks to results available on asymptotic expansions of Bergman kernel \cite{MaMar07} and Toeplitz operators \cite{MaMar10}, what one can prove is indeed the following result, which obviously implies \eqref{maineq}.
\begin{thm}\label{mainthm}
There is a complete asymptotic expansion
$$ T^*_m \circ \Delta_m \circ T_m(f) 
= \sum_{r \geq 0} P_r(f) m^{n-1-r} + O(m^{-\infty}), $$
where $P_r$ are self-adjoint differential operators on $C^\infty(M)$.
More precisely, for any $k,R \geq 0$ there exist constants $C_{k,R,f}$ such that
$$ \left\| T_m^* \circ \Delta_m \circ T_m (f) - \sum_{r=0}^R P_r(f) m^{n-1-r} \right\|_{C^k(M)} \leq C_{k,R,f} m^{n-R-2}. $$
Moreover one has
\begin{equation*}
P_0(f) = \Delta f, \qquad P_1(f) = -\frac{1}{2\pi}\Delta^2 f.
\end{equation*}
\end{thm}

The construction of the quantized Laplacian $\Delta_m$ was inspired by a work of J. Fine on the Hessian of the Mabuchi energy \cite{Fin10}.
Even though in principle $\Delta_m$ is unrelated to the problem of finding canonical metrics on $M$, when $\omega$ is balanced in the sense of Donaldson (see definition recalled at Section \ref{sec::bal}) the relation between $\Delta_m$ and $\Delta$ is even more evident as shown by the following
\begin{thm}\label{thm::bal}
If $\omega$ is $m$-balanced then
$$ \Delta_m (A) = C\, T_m \circ \Delta \circ T_m^*(A) $$
for al $A \in V_m$, where $C=\frac{m^{n-1} \left(\int_M \frac{\omega^n}{n!}\right)^2}{\left(\dim H^0(M,L^m)\right)^2}$.
\end{thm}

\bigskip

Thanks to A. Ghigi and A. Loi for some useful discussion on Berezin-Toeplitz quantization. The main part of these note has been written in 2010 while the author was visiting Princeton University, whose hospitality is gratefully acknowledged. At that time the author was partially supported by a Marie Curie IOF (program CAMEGEST, proposal no. 255579). A recent pre-print of J. Keller, J. Meyer, and R. Seyyedali has a substantial overlapping with the present work \cite{KelMeySey2015}. The author became aware of that pre-print when it appeared on the arXiv.

\section{The space $V_m$}\label{sec::V_m}

The space $V_m$ is nothing but $\End(H_m)$, being $H_m=H^0(M,L^m)$ the space of the holomorphic sections of $L^m$. By Riemann-Roch theorem $\dim V_m$ grows like a positive multiple of $m^{2n}$ when $m \to \infty$.
The space $H_m$ is equipped with a Hermitian inner product $b_m$ induced by the Hermitian metric $h^m$ on $L^m$ and the K\"ahler form $\omega$. Explicitly it is given by
\begin{equation}\label{scalprod}
b_m (s,t) = \int_M h^m(s,t) \frac{\omega^n}{n!}
\end{equation}
for all $s,t \in H_m$. 
Thus $V_m$ is a Hermitian vector space with inner product defined by
\begin{equation}\label{HermprodV_m}
\langle A, B \rangle = \tr(AB^*),
\end{equation}
for all $A,B \in V_m$.
Here $B^*$ denotes the adjoint of $B$ with respect to $b_m$.

\section{The maps $T_m$ and $T^*_m$}

The map $T_m: C^\infty(M) \to V_m$ is the well known Berezin-Toeplitz quantization operator \cite{MaMar07}.
Given a smooth function $f$ on $M$, the operator $T_m(f)$ is the composition $T_m(f) = P \circ M(f)$, where $M(f): H_m \to \Gamma(M,A^m)$ is the multiplication by $f$: 
$$ M(f)(s) = fs,$$ and $P: \Gamma(M,A^m) \to H_m$ is the orthogonal projection with respect to the obvious extension of the inner product $b_m$ to smooth sections.

The space of smooth function $C^\infty(M)$ is equipped with the $L^2$-product induced by $\omega$, given by
\begin{equation}
\langle f , g \rangle = \int_M f \bar g \frac{\omega^n}{n!},
\end{equation}
for all $f , g \in C^\infty(M)$.
Let $T^*_m: V_m \to C^\infty(M)$ be the adjoint of $T_m$.

\begin{lem}\label{adjT}
Let $\{s_\alpha\}$ be an orthonormal basis of $H_m$. For all $A \in V_m$ it holds:
$$ T^*_m(A) = \sum_\alpha h^m(As_\alpha, s_\alpha).$$
\end{lem}
\begin{proof}
It is an easy consequence of general theory. 
For every $f \in C^\infty(M)$ one has
$$ \tr (A \,T_m(f)^*) = \sum_\alpha b_m\left( A s_\alpha , T_m(f) s_\alpha \right).$$ 
Substituting
$$ T_m(f) s_\alpha  = \sum_\beta \left( \int_M f h^m(s_\alpha,s_\beta) \frac{\omega^n}{n!} \right) s_\beta, $$
it follows
$$ \tr (A \, T_m(f)^*) = \sum_\alpha \int_M \bar f(x) h^m(As_\alpha,s_\alpha) \frac{\omega^n}{n!}, $$ 
which gives the thesis by arbitrariness of $f$ after noting that
$$ \int_M T^*_m(A) \bar f \frac{\omega^n}{n!} = \tr (A \, T_m(f)^*).$$ 
\end{proof}

Note that the map $T^*_m$ takes an endomorphisms $A \in V_m$ to the restriction to the diagonal of its integral kernel.
More precisely, given an orthonormal basis $\{s_\alpha\}$ of $H_m$, the integral kernel of $A$ is the smooth section $ K(A)$ of $L^m \boxtimes L^{-m}$ over $M \times M$ given by
$$ K(A)(x,y) = \sum_{\alpha,\beta} \int_M h^m(As_\alpha,s_\beta)(z) s_\beta(y) \otimes s_\alpha^*(x) \frac{\omega^n_z}{n!},$$
where $s_\alpha^*(x)$ is the metric dual of $s_\alpha(x)$ in the fiber of $L^m$ over the point $x$.
The restriction of the kernel to the diagonal is (naturally identified with) the smooth function $T^*_m(A)$ thanks to Lemma \ref{adjT}.
When $A$ is of the form $T_m(f)$ for some smooth function $f$, the integral kernel is given by
$$ K(T_m(f))(x,y)= \sum_{\alpha,\beta} \int_M f(z) h^m (s_\alpha,s_\beta)(z) s_\beta(y) \otimes  s_\alpha^*(x) \frac{\omega^n_z}{n!},$$
whence
$$ T^*_m \circ T_m(f)(x) = \sum_{\alpha,\beta} \int_M f(z) h^m (s_\alpha,s_\beta)(z) h^m(s_\beta, s_\alpha)(x) \frac{\omega^n_z}{n!}.$$
For a constant function $f=c \in \bold R$, one has $$ T^*_m \circ T_m (c) = c \, \rho_m,$$ where $\rho_m = \sum_\alpha |s_\alpha|_{h^m}^2$ is the so-called Bergman kernel of $\omega$. 

\section{The operator $\Delta_m$}\label{sec::Delta_m}

The operator $\Delta_m : V_m \to V_m$ is a self-adjoint operator which depends just on projective geometry of $M$ in $\bold P(H_m)$.
Consider the embedding 
$$\iota_m : M \to \bold P (H_m),$$ 
given by the Kodaira map of $M$ in $\bold P(H_ m^*)$ induced by $L^m$, followed by the isomorphism $\bold P(H_ m^*) \simeq \bold P(H_ m)$ induced by the Hermitian product $b_m$.
Every endomorphism $A$ of $H_m$ induces a (holomorphic) vector field $\nu(A)$ on $\bold P(H_m)$ whose flow is given by 
$$ \Phi_{\nu(A)}^t (z) = e^{tA}z. $$ 
Let $\Lambda_m$ be the hyperplane bundle on $\bold P(H_m)$, endowed with the Hermitian metric induced by $b_m$, and let $g_m$ be the pull-back to $M$ of the associated Fubini-Study metric on $\bold P(H_m)$.
One can restrict $\nu(A)$ to $M$ as a section of $\iota_m^*T\bold P(H_m)$, and then project orthogonally to $TM \subset \iota_m^*T\bold P(H_m)$ to get a smooth vector field $e_m(A)$ on $M$.
This defines a map
$$ e_m: V_m \to \Gamma(TM).$$
Recall that $V_m$ has an inner product defined by \eqref{HermprodV_m}. On the other hand, $\Gamma(TM)$ is equipped with the $L^2$-inner product induced by the K\"ahler metric $g_m$: 
$$ (\eta,\xi)_m = \int_M g_m(\eta,\xi) \frac{\omega_m^n}{n!}, $$
for all $\eta , \xi \in \Gamma(TM)$ (here $\omega_m$ is the K\"aler form of $g_m$, i.e. the pull-back of the Fubini-study form to $M$).
Thus one can form the adjoint operator 
$$ e_m^* : \Gamma(TM) \to V_m,$$ and finally define 
\begin{equation}
\Delta_m = e_m^* \circ e_m.
\end{equation}

Next lemma shows that the vector field $e_m(A)$ and the function $T^*_m(A)$ are related through the projectively induced K\"ahler metric $g_m$.

\begin{lem}\label{emgrad}
For all $A\in V_m$ one has 
$$ e_m(A) = \grad_m \frac{T^*_m(A)}{\rho_m}, $$
where the gradient is taken with respect to the Riemannian metric $g_m$.
\end{lem}

\begin{proof}
We have to show that $g_m \left(e_m(A),v\right) = v \left( T^*_m(A)/\rho_m \right) $ for all vector field $v \in \Gamma(TM)$.
In order to do this, consider a smooth extension $\tilde v$ of $v$ to a smooth vector field of $\bold P(H_m)$.
Since $g_m$ is induced by the Fubini-Study metric $g_{FS}$ on $\bold P(H_m)$, and $e_m(A)$ is the orthogonal projection of $\nu(A)$ on $TM$, one has 
\begin{equation}\label{gmemA}
g_m(e_m(A),v) = \iota_m^* \, g_{FS} (\nu(A), \tilde v).
\end{equation}   

The right hand side of the equation above can be related to a function on $\bold P(H_m)$ naturally associated to $A$.
Indeed we claim that $\nu(A)$ is the gradient of the function $\mu_A$ defined by
$$ \mu_A(s) = \frac{b_m\left(As, s\right)}{b_m(s,s)}.$$ This is quite standard, but a proof of that fact is included at the end of the proof for convenience of the reader. 
Now we go ahead taking the claim for grant.
From \eqref{gmemA} one gets
$$ g_m (e_m(A),v) = v (\iota_m^* \mu_A), $$
thus it remains to prove the identity
\begin{equation}\label{muA=K/berg}
\iota_m^* \mu_A = T^*_m(A)/\rho_m.
\end{equation}
To this end, let $\{s_\alpha\}$ be an orthonormal basis of $H_m$, so that the pull-back of $\mu_A$ to $M$ is given by 
$$ \iota_m^* \mu_A(x) = \frac{\sum_{\alpha, \beta} s_\alpha(x) \overline{s_\beta(x)} b_m(As_\alpha,s_\beta) }{\sum_\gamma |s_\gamma(x)|^2},$$ 
where the ratio $\frac{s_\alpha(x) \overline{s_\beta(x)}}{\sum_\gamma |s_\gamma(x)|^2}$ is well defined and can be computed choosing an arbitrary Hermitian metric on the line bundle $L^m$.
In particular, taking $h^m$ it becomes $ \frac{h^m(s_\alpha,s_\beta)(x)}{\sum_\gamma |s_\gamma|^2_{h^m}(x)}$, whence 
$$ \iota_m^* \mu_A(x) = \frac{\sum_{\alpha, \beta} \int_M h^m(As_\alpha,s_\beta)(z) h^m(s_\alpha, s_\beta)(x) \frac{\omega_z^n}{n!}}{\sum_\gamma |s_\gamma|^2_{h^m}(x)},$$ 
and the identity \eqref{muA=K/berg} follows by definition of $\rho_m$ and Lemma \ref{adjT}.

Finally, in order to prove the claim above, let $(z_\alpha)$ be homogeneous coordinates on $\bold P(H_m)$ corresponding to the basis $\{s_\alpha\}$. 
The function $\mu_A$ then takes the form 
$$ \mu_A(z) = \frac{\bar z A z^t}{|z|^2},$$ 
where now $A=\left(A_{\alpha \beta}\right)$ denotes the matrix that represents the endomorphism $A$ with respect the chosen basis. 
The equality between $\nu(A)$ and the gradient of $\mu_A$ can be proved in local affine coordinates, but here we consider the projection of $H_m\setminus \{0\}$ on $\bold P(H_m)$, and the fact that $\nu(A)$, $g_{FS}$ and $\mu^A$ lift to $\bold C^*$-invariant objects (which will be denotes with the same symbols). 
In particular one has 
$$ \nu(A) = \sum_{\alpha,\beta}A_{\alpha\beta} \left(z_a \frac{\partial}{\partial z_\beta} + \bar z_\beta \frac{\partial}{\partial \bar z_\alpha}\right),$$
and 
$$g_{FS}=\sum_i \frac{dz_id\bar z_i}{|z|^2} - \sum_{i,j} \frac{\bar z_i z_j dz_i d \bar z_j}{|z|^4},$$
whence
$$ i_{\nu(A)} g_{FS} = \sum_{\alpha,\beta}A_{\alpha_\beta} \left( \frac{z_\alpha d \bar z_\beta + \bar z_\beta d z_\alpha}{|z|^2} - \frac{z_\alpha\bar z_\beta d |z|^2}{|z|^4}\right) = d \mu_ A, $$
which proves the claim.
\end{proof}

Next lemma characterizes the kernel of $\Delta_m$.

\begin{lem}
$\Delta_m (A)=0$ if and only if $A$ is a multiple of the identity.
\end{lem}
\begin{proof}
By definition $\Delta_m = e_m^* \circ e_m$, and by Lemma \ref{emgrad} and its proof it follows $e_m(A) = \grad_m \frac{T^*_m(A)}{\rho_m}$.
Thus $\Delta_m(A)=0$ if and only if $\frac{T^*_m(A)}{\rho_m} = c$ for some $c \in \mathbf C$. Let $I \in V_m$ be the identity. The identity $T_m^*(I) = \rho_m$ implies $\Delta_m(A)=0$ if and only if $A-cI \in \ker T_m^*$, thus the thesis follows by injectivity of $T_m^*$ \cite[Proposition 3.4]{Sch12}.

Alternatively, one can argue more geometrically as follows. In the proof of Lemma \ref{emgrad} has been introduced a smooth function $\mu_A$ on $\mathbf P(H_m)$  satisfying $\frac{T^*_m(A)}{\rho_m} = \iota_m^* \mu_A$.
Thus by Lemma \ref{emgrad} one has $\Delta_m(A)=0$ if and only if $\iota_m^* d \mu_A=0$.
Then the thesis follows by showing that the locus where $d \mu_A=0$ contains no positive dimensional holomorphic submanifolds (or, in other words, $\ker d\mu_A$ is totally real), unless $\mu_A$ is constant.
\end{proof}

Now we pass to give a more explicit description of the operator $\Delta_m$.
To this end fix an orthonormal basis $\{s_\alpha\}$ of $H_m$ and let $(z_i)$ be the corresponding homogeneous coordinates on $\mathbf P(H_m)$. Moreover this identifies $V_m$ with the space of $\dim H_m \times \dim H_m$ complex matrices.
Consider the map 
$$ \Psi_m : \bold P(H_m) \to V_m$$
defined by $\Psi_m(z) = \frac{z\bar z^t}{|z|^2}$.
Note that by Lemma \ref{adjT} follows that
\begin{equation}\label{pbtrAPsi}
\iota_m^* \tr(A \Psi_m) = \frac{T_m^*(A)}{\rho_m}
\end{equation}
for all $A \in V_m$.
On the other hand, by definition of $\Delta_m$ one has
$$ \tr(\Delta_m(A)B^*) = \int_M g_m (e_m(A),e_m(B)) \frac{\omega_m^n}{n!},$$ 
thus by Lemma \ref{emgrad} together with \eqref{pbtrAPsi} one gets
$$ \int_M g_m (e_m(A),e_m(B)) \frac{\omega_m^n}{n!} = \int_M i\partial \tr(\Psi_mA) \wedge \bar \partial \tr(\Psi_mB^*) \wedge \frac{\omega_m^{n-1}}{(n-1)!} $$ 
whence 
$$ \tr(\Delta_m(A)B^*) = \int_M i \partial \tr(\Psi_mA) \wedge \bar \partial \tr(\Psi_mB^*) \wedge \frac{\omega_m^{n-1}}{(n-1)!}. $$ 
Let 
$$ \Phi_m : \bold P(H_m) \to V_m^*$$
be the map obtained by composing $\Phi$ with the dual paring induced by the Hermitian metric $b_m$ on $V_m$. More explicitiely one has
$$ \Phi_m (z) (A) = \tr(\Psi_m(z)A) $$ 
for all $A \in V_m$ and $z \in \mathbf P(H_m)$.
By computation above we proved the following

\begin{prop}\label{prop::Deltaintform}
Consider the $\End(V_m)$-valued differential form on $\bold P(H_m)$ defined by 
$$ \Xi_m = i \partial \Phi_m \wedge \bar \partial \Psi_m \wedge e^{\omega_{FS}}.$$
Then it holds $$ \Delta_m = \int_M \Xi_m.$$ 
\end{prop}

Here $e^{\omega_{FS}}$ is a mixed-degree form defined by the exponential series. Since $\omega_{FS}^k=0$ for all $k \geq \dim H_m$, one has
$$e^{\omega_{FS}} = 1 + \omega_{FS} + \frac{\omega_{FS}^2}{2} + \dots + \frac{\omega_{FS}^{\dim H_m - 1}}{(\dim H_m -1)!}.$$ This implies that $\Xi_m$ has mixed degree. More interestingly it depends just on the dimension of $\mathbf P(H_m)$ (and on choice of homogeneous coordinates) and it is independent of $M$.

\begin{cor}
$$ \tr (\Delta_m) = 2 \pi n \, m^n \int_M \frac{\omega^n}{n!}.$$ \end{cor}
\begin{proof}
Recall that we identified $V_m$ with the space of $\dim H_m \times \dim H_m$ matrices by choosing an orthonormal basis $\{s_\alpha\}$ of $H_m$.
The set of canonical matrices $E_{ij}$, then form an orthonormal basis of $V_m$.
Thus by Proposition \ref{prop::Deltaintform} one has
\begin{eqnarray*}
\tr(\Delta_m)
&=& \sum_{\alpha,\beta} \langle \Delta_m(E_{\alpha \beta}) , E_{\alpha \beta} \rangle \\
&=& \sum_{\alpha,\beta} \int_M i \partial \left(\frac{z_\alpha \bar z_\beta}{|z|^2}\right) \wedge \bar \partial \left(\frac{z_\beta \bar z_\alpha}{|z|^2}\right) \wedge e^{\omega_{FS}} \\
&=& \int_M \left(\frac{i \partial \bar \partial |z|^2}{|z|^2} - \frac{ i \partial |z|^2 \wedge \bar \partial |z|^2 }{|z|^4} \right) \wedge e^{\omega_{FS}} \\
&=& 2 \pi \int_M \omega_{FS} \wedge e^{\omega_{FS}} \\
&=& 2 \pi n \int_M \frac{\omega_m^n}{n!},
\end{eqnarray*}
whence the thesis follows since $\omega_m$ is cohomologous to $m \omega$.
\end{proof}
\section{Proof of Theorem \ref{mainthm}}

First of all we recall some results on asymptotic expansions in Berezin-Toeplitz quantization.
\begin{thm}
There is a sequence $\{b_r\}$ of self-adjoint differential operators acting on $C^\infty(M)$ such that for any smooth function $f \in C^\infty (M)$ one has the asymptotic expansion
\begin{equation}\label{expbr}
T_m^* \circ T_m (f) = \sum_{r \geq 0} b_r(f) m^{n-r} + O(m^{-\infty}),
\end{equation}
and for any $k,R \geq 0$ there exist constants $C_{k,R,f}$ such that
$$ \left\| T_m^* \circ T_m (f) - \sum_{r=0}^R b_r(f) m^{n-r} \right\|_{C^k(M)} \leq C_{k,R,f} m^{n-R-1}. $$
Moreover one has
\begin{eqnarray*}
b_0(f) &=& f, \\
b_1(f) &=& \frac{\scal(g)}{8\pi} f - \frac{1}{4\pi} \Delta f. \\
\end{eqnarray*}
\end{thm}

\begin{proof}
See Ma and Marinescu \cite{MaMar07}. The only fact one still needs to show is self-adjointness of operator $b_r$. It follows readily by self-adjointness of $T_m^* \circ T_m$ and expansion \eqref{expbr}. Indeed one has 
$$ 0 = \sum_{r = 0}^R m^{n-r}  \int_M \left( b_r(f) \, \bar g - f \, \overline{b_r(g)} \right) \frac{\omega^n}{n!} + O(m^{n-R-1}),$$
as $m \to + \infty$, for all $f,g \in C^\infty(M)$. 
\end{proof}

Since the Bergman kernel satisfies $\rho_m = T_m^* \circ T^m(1)$, one recovers the well known asymptotic expansion \cite{Tia90, MaMar07, Ghi10, MaMar10}
\begin{equation}\label{exprhom} \rho_m = \sum_{r \geq 0} a_r m^{n-r} + O(m^{-\infty}), \end{equation}
where $a_r = b_r(1) \in C^\infty(M)$ depends polynomially in the curvature of $g$ and its covariant derivatives. In particular
\begin{equation} 
a_0 = 1, \qquad a_1 = \frac{\scal(g)}{8\pi}.
\end{equation}

\begin{lem}\label{lem::Deltamfunction}
For any $A \in V_m$ one has
\begin{equation*}
\Delta_m(A) = T_m \left( \frac{\omega_m^n}{\rho_m \, \omega^n} \Delta_{g_m} \left( \frac{T_m^* (A)}{\rho_m} \right) \right),
\end{equation*}
where $\Delta_{g_m}$ denotes the Laplacian of metric $g_m$.
\end{lem}
\begin{proof}
By definition of $\Delta_m$, for any $B \in V_m$ it holds
\begin{equation*}
\tr(\Delta_m(A)B^*) = \int_M g_m (e_m (A), e_m (B)) \frac{\omega^n_m}{n!}, 
\end{equation*}
whence, by Lemma \ref{emgrad} and integration by parts it follows
\begin{equation*} 
\tr(\Delta_m(A)B^*) = \int_M \Delta_{g_m} \left( \frac{T_m^* (A)}{\rho_m} \right) \overline{ \frac{T_m^* (B)}{\rho_m} } \frac{\omega_m^n}{n!}. \end{equation*}
The right hand side can be rewritten as
\begin{equation*}
\int_M \frac{\omega_m^n}{\rho_m \, \omega^n}\Delta_{g_m} \left( \frac{T_m^* (A)}{\rho_m} \right) \overline{ T_m^* (B) } \frac{\omega^n}{n!} = \tr \left(T_m \left( \frac{\omega_m^n}{\rho_m \, \omega^n} \Delta_{g_m} \left( \frac{T_m^* (A)}{\rho_m} \right) \right) B^* \right),
\end{equation*} 
whence the statement follows by arbitrariness of $B$.
\end{proof}

For any $f \in C^\infty(M)$, by Lemma above one has
\begin{equation}\label{perpasymptexp}
T^*_m \circ \Delta_m \circ T_m (f) 
= T_m^* \circ T_m \left( \frac{\omega_m^n}{\rho_m \, \omega^n} \Delta_{g_m} \left( \frac{T_m^* \circ T_m (f)}{\rho_m} \right) \right),
\end{equation}
thus the statement of Theorem \ref{mainthm} follows readily by Theorem \ref{expbr} and asymptotic expansion \eqref{exprhom}.
In particular one has
\begin{equation*}
T_m^* \circ T_m (f) = m^n f + m^{n-1} b_1(f) + O(m^{n-2})
\end{equation*}
whence
\begin{eqnarray*}
\rho_m &=& m^n + m^{n-1} a_1 + O(m^{n-2}), \\
\frac{T_m^* \circ T_m (f)}{\rho_m} &=& 
f + m^{-1} (b_1(f) -a_1f) + O(m^{-2}) \\
\omega_m &=& 
m \omega + O(m^{-1}), \\
\frac{\omega_m^n}{\rho_m \omega^n} 
&=& 
1 - m^{-1} a_1 + O(m^{-2}), \\
\Delta_{g_m} (f) &=& m^{-1} \Delta(f) + O(m^{-3}).
\end{eqnarray*}
Substituting in \eqref{perpasymptexp} finally gives
\begin{eqnarray*}
T^*_m \circ \Delta_m \circ T_m (f) 
&=& T_m^* \circ T_m \left( \left(1 - m^{-1} a_1 \right) m^{-1}\Delta \left( f + m^{-1} (b_1(f) -a_1f) \right) + O(m^{-3}) \right) \\
&=& m^{-1} T_m^* \circ T_m \left( \Delta (f) + m^{-1} \left(\Delta b_1(f) - \Delta(a_1f) - a_1\Delta(f) \right) + O(m^{-2}) \right) \\
&=& m^{-1} T_m^* \circ T_m \left( \Delta (f) + m^{-1} \left(\Delta b_1(f) - \Delta(a_1f) - a_1\Delta(f) \right) + O(m^{-2}) \right) \\
&=& m^{n-1} \Delta  f + m^{n-2} \left( \Delta (b_1(f)) + b_1 ( \Delta (f) ) - \Delta(a_1f) - a_1\Delta(f) \right) + O(m^{n-3}).
\end{eqnarray*}
This obviously proves $P_0 = \Delta$ and, recalling that $b_1(f) = \frac{\scal(g)}{8\pi} f - \frac{1}{4\pi} \Delta (f)$ and $a_1 = \frac{\scal(g)}{8\pi}$, it gives
\begin{eqnarray*}
8 \pi \, P_1(f) 
&=& \Delta (\scal(g) f - 2 \Delta(f)) + \scal(g) \Delta (f) - 2 \Delta^2(f) - \Delta(\scal(g)f) - \scal(g) \Delta(f) \\
&=& - 4 \Delta^2(f),
\end{eqnarray*}
which concludes the proof.

\section{Balanced metrics}\label{sec::bal}

Balanced metrics have been introduced by Donaldson in connection with the existence problem of constant scalar curvature K\"ahler metric on polarized manifolds \cite{Don01}.
Recall that a metric is called \emph{$m$-balanced} if the density of state function $\rho_m$ is constant.
Note that the value of such a constant is not arbitrary for $\rho_m$ satisfies $\int_M \rho_m \frac{\omega^n}{n!} = \dim H^0(M,L^m)$. 
Moreover, since in general one has $\omega_m = m \omega + \frac{i}{2\pi} \partial \bar \partial \log \rho_m$, $\omega$ is $m$-balanced if and only if $\omega_m = m \omega$.
Thus, assuming that $\omega$ is $m$-balanced, by Lemma \ref{lem::Deltamfunction} for any $A \in V_m$ one has
\begin{eqnarray*}
\Delta_m(A) &=& 
\rho_m^{-2} \, T_m \left( \frac{\omega_m^n}{ \omega^n} \Delta_{g_m} \left(T_m^* (A) \right) \right) \\
&=&
\frac{m^{n-1} \left(\int_M \frac{\omega^n}{n!}\right)^2}{\left(\dim H^0(M,L^m)\right)^2} \, T_m \circ \Delta \circ T_m^* (A),
\end{eqnarray*}
which proves Theorem \ref{thm::bal}.

\end{document}